\documentclass[10pt,reqno]{amsart}
\usepackage{amsmath}
\usepackage{amsthm}
\usepackage{amssymb}
\usepackage{graphicx}
\usepackage{amsfonts}
\usepackage{latexsym, wasysym}
\usepackage{hyperref,subcaption}

\graphicspath{{IMAGES/}}

\usepackage{enumitem}
\setlist[itemize]{leftmargin=*}
\setlist[enumerate]{leftmargin=*}

\usepackage[font=footnotesize,labelfont=bf]{caption}
\usepackage[font=scriptsize,labelfont=bf]{subcaption}

\usepackage[table]{xcolor}
\usepackage{cite}
\usepackage{multicol}

\theoremstyle{plain}
\newtheorem{theorem}{Theorem}

\newtheorem{lemma}[theorem]{Lemma}

\theoremstyle{definition}
\newtheorem{question}[theorem]{Question}

\newtheorem{remark}[theorem]{Remark}

\newcommand{\Q}{\mathbb{Q}}
\newcommand{\R}{\mathbb{R}}
\newcommand{\N}{\mathbb{N}}

\newcommand{\K}{\mathbb{K}}
\renewcommand{\L}{\mathbb{L}}

\newcommand{\C}{\mathbb{C}}
\newcommand{\PP}{\mathfrak{p}}
\renewcommand{\O}{\mathcal{O}}

\newcommand{\I}{\mathcal{I}}

\newcommand{\Ei}{\operatorname{Ei}}

\renewcommand{\Re}{\operatorname{Re}}

\let\oldenumerate=\enumerate
	\def\enumerate{
	\oldenumerate
	\setlength{\itemsep}{5pt}
	}
\let\olditemize=\itemize
	\def\itemize{
	\olditemize
	\setlength{\itemsep}{5pt}
	}

\allowdisplaybreaks

\begin{document}

\title[Unconditional Explicit Mertens' Theorems for Number Fields]{Unconditional Explicit {M}ertens' Theorems for Number Fields and {D}edekind Zeta Residue Bounds}

	\author[S.R.~Garcia]{Stephan Ramon Garcia}
	\address{Department of Mathematics, Pomona College, 610 N. College Ave., Claremont, CA 91711, USA} 
	\email{stephan.garcia@pomona.edu}
	\urladdr{\url{http://pages.pomona.edu/~sg064747}}
	
\author[E.S.~Lee]{Ethan Simpson Lee}
\address{School of Science, UNSW Canberra at the Australian Defence Force Academy, Northcott Drive, Campbell, ACT 2612} 
\email{ethan.s.lee@student.adfa.edu.au}
\urladdr{\url{https://www.unsw.adfa.edu.au/our-people/mr-ethan-lee}}
	
\thanks{SRG supported by NSF Grant DMS-1800123.}



\begin{abstract}
We obtain unconditional, effective number-field analogues of the three Mertens' theorems, all with explicit constants and valid for $x\geq 2$.  Our error terms are explicitly bounded in terms of the degree and discriminant of the number field.  To this end, we provide unconditional bounds, with explicit constants, for the residue of the corresponding Dedekind zeta function at $s=1$.
\end{abstract}

\maketitle
\section{Introduction}
In 1874, twenty-two years before the proof of the prime number theorem \cite{hadamard1896distribution, de1896fonction}, 
Mertens \cite{Mertens} proved the following three results
\begin{align*}
\sum_{p\leq x} \frac{\log p}{p} &= \log x + O(1), \\
\sum_{p \leq x} \frac{1}{p} &= \log \log x + M + O\bigg(\frac{1}{\log x} \bigg), \\
\prod_{p\leq x} \bigg(1 -\frac{1}{p} \bigg) &= \frac{e^{-\gamma}}{\log x} \big(1+o(1) \big),
\end{align*}
collectively referred to as Mertens' theorems.
Here $p$ denotes a rational prime number, $M = 0.2614\ldots$ is the Meissel--Mertens constant, and $\gamma = 0.5772\ldots$ is the Euler--Mascheroni constant. 
Proofs can be found in Ingham \cite[Thm.~7]{Ingham} and Montgomery--Vaughan \cite[Thm.~2.7]{MontgomeryVaughan}. Rosser--Schoenfeld \cite[(3.17) - (3.30)]{Rosser} provide unconditional error terms with explicit constants.

Rosen \cite[Lem.~2.3, Lem.~2.4, Thm.~2]{Rosen} generalized Mertens' theorems to the number-field setting, but without explicit constants in the error terms (see also Lebacque's paper \cite{Lebacque}). Assuming the Generalized Riemann Hypothesis, 
the authors obtained effective number-field analogues of Mertens' theorems, 
in which the implied constants are explicit in their dependence upon the degree and discriminant of the number field \cite{EMT4NF1}.
We now approach the same family of problems unconditionally; that is, without assuming any unproved conjectures.

\subsection*{Definitions}
Let $\K$ denote a number field of degree $n_{\K} = [\K : \Q]$,
with ring of algebraic integers $\O_{\K}$.
Let $\Delta_{\K}$ denote the discriminant of $\K$
and let $N(\mathfrak{a})$ denote the norm of an ideal $\mathfrak{a}\subset\O_{\K}$;
we let $\PP$ denote a prime ideal of $\O_{\K}$.
Note that $|\Delta_{\K}| \geq 3$ for $n_{\K} \geq 2$.
The Dedekind zeta function 
\begin{equation*}
\zeta_{\K}(s) = \sum_{\mathfrak{a} \subseteq \O_{\K}} \frac{1}{N(\mathfrak{a})^s}
= \prod_{\PP} \bigg(1 - \frac{1}{N(\PP)^s}\bigg)^{-1}
\end{equation*}
is analytic on $\Re{s} > 1$ and extends meromorphically to $\C$, except for a simple pole at $s=1$. 
By the analytic class number formula, the residue of $\zeta_{\K}(s)$ at $s=1$ is 
\begin{equation}\label{eqn:residue_class_ana_form}
\kappa_{\K} = \frac{2^{r_1}(2\pi)^{r_2}h_{\K}R_{\K}}{w_{\K}\sqrt{|\Delta_{\K}|}},
\end{equation}
in which $r_1$ is the number of real places of $\K$, $r_2$ is the number of complex places of $\K$, $w_{\K}$ is the number of roots of unity in $\K$, $h_{\K}$ is the class number of $\K$, and $R_{\K}$ is the regulator of $\K$ \cite{Lang}.

The nontrivial zeros of $\zeta_{\K}$ lie in the critical strip, $0 < \Re s < 1$, where there might exist an 
exceptional zero $\beta$, which is real and cannot lie too close to $\Re{s} =1$ \cite[p.~148]{Stark}. There are some cases in which 
$\beta$ is known to not exist. For example, Heilbronn \cite{Heilbronn} (later generalized by Stark \cite{StarkBS}) showed 
that if $\L$ is a normal extension of $\K$ and $\L$ has no quadratic subfield, then $\beta$ does not exist.
The Generalized Riemann Hypothesis (GRH) asserts that the nontrivial zeros of $\zeta_{\K}(s)$ satisfy 
$\Re s = \frac{1}{2}$ and that the exceptional zero $\beta$ does not exist.

\subsection*{Statement of results}
Our main result (Theorem \ref{Theorem:Main} below) is an effective version of Mertens' theorems for number fields.
It is unconditional and the error terms depend explicitly only upon the two
easily-obtained parameters $n_{\K}$ and $\Delta_{\K}$;
see Remark \ref{Remark:Kappa}.
Moreover, our estimates are valid for all $x \geq 2$.
Our overall roadmap follows Diamond--Halberstam \cite[p.~128-9]{DiamondHalberstam}, 
although significant adaptations and technical lemmas are required to make things explicit.\footnote{Another possible approach might be to adapt Hardy's method \cite{HardyNote1927,HardyNote2935,BardestaniFreiberg}.}

\begin{theorem}\label{Theorem:Main}
Let $\K$ be a number field with $n_{\K} \geq 2$.  Then for $x \geq 2$,
\begin{align}
    \sum_{N(\PP)\leq x}\frac{\log{N(\PP)}}{N(\PP)} &\,=\, \log{x} + A_{\K}(x) ,\label{eq:A1}\tag{A1}\\[5pt]
    \sum_{N(\PP)\leq x}\frac{1}{N(\PP)} &\,=\, \log\log{x} + M_{\K} + B_{\K}(x),\label{eq:B1} \tag{B1}\\[5pt]
    \prod_{N(\PP)\leq x}\left(1 - \frac{1}{N(\PP)}\right) &\,=\, \frac{e^{-\gamma}}{\kappa_{\K}\log{x}} \big(1 + C_{\K}(x) \big),\label{eq:C1}\tag{C1}
\end{align}
in which
\begin{align}
M_{\K} \,&=\, \gamma + \log{\kappa_{\K}} + \sum_{\PP}\left[\frac{1}{N(\PP)} + \log\left(1 - \frac{1}{N(\PP)}\right)\right]\label{eq:M1}\tag{M1},\\[5pt]
|A_{\K}(x)| 
&\leq  \Upsilon_{\K} ,
\label{eq:A2}\tag{A2}\\[5pt]
|B_{\K}(x)| &\leq \frac{2 \Upsilon_{\K}}{\log x}, \label{eq:B2}\tag{B2}\\[5pt]
|C_{\K}(x)| &\,\leq\, |E_{\K}(x)|e^{| E_{\K}(x)|} \qquad \text{with}\qquad |E_{\K}(x)| \,\leq\, \frac{n_{\K}}{x-1} + |B_{\K}(x)|,\label{eq:C2}\tag{C2} 
\end{align}
and
\begin{align}
\Lambda_{\K} 
&= e^{28.2n_{\K}+5} (n_{\K}+1)^{ \frac{5(n_{\K}+1)}{2}} |\Delta_{\K}|^{\frac{1}{n_{\K}+1}} 
(\log |\Delta_{\K}|)^{n_{\K}}, \label{eq:Lambda} \tag{$\Lambda$}\\[5pt]
\Upsilon_{\K}&=  \left(  \frac{(n_{\K}+1)^2}{2\kappa_{\K}(n_{\K}-1)}\Lambda_{\K}   + 1\right)
+  \frac{0.55\, \Lambda_{\K}n_{\K}(n_{\K}+1)}{\kappa_{\K}} + n_{\K} + 
 40.31 \dfrac{\Lambda_{\K}n_{\K}}{\kappa_{\K}} .\label{eq:Upsilon}\tag{$\Upsilon$}
\end{align}
In particular, $E_{\K}(x) = o(1)$, hence $C_{\K}(x) = o(1)$ as $x \to \infty$.  Furthermore,
\begin{equation}\label{eq:M2}\tag{M2}
    \gamma + \log \kappa_{\K} -n_{\K} \,\leq\, M_{\K} \,\leq\,  \gamma + \log \kappa_{\K}.
\end{equation}
All quantities above can be effectively bounded, with explicit constants, in terms of $|\Delta_{\K}|$ and $n_{\K}$ alone;
see the remarks below.
\end{theorem}

In order to prove Theorem \ref{Theorem:Main} with error bounds not dependent upon a 
potential exceptional zero of $\zeta_{\K}$, our proof eschews estimates of the prime-ideal counting function, such as \cite[Cor.~1]{Grenie2019}, in favor of an alternative. 
We appeal instead to a result of Sunley (see Theorem \ref{Theorem:Sunley})
for an explicit estimate for the ideal-counting function for $\K$ that does not require information about the zeros of $\zeta_{\K}$.

\begin{remark}
For $n_{\K} = 1$, that is $\K = \Q$, our approach provides weaker error bounds than
Rosser--Schoenfeld \cite[Thms.~5-7]{Rosser}.
Much more is known about the Riemann zeta function than
a generic Dedekind zeta function, so this is not surprising.
\end{remark}

\begin{remark}\label{rem:MainImproved}
For $n_{\K} = 2$ and $n_{\K}=3$, one can obtain slightly improved bounds
by implementing \eqref{eq:BetterBit} in the proof of \eqref{eq:B2} throughout
the subsequent computations.
\end{remark}

\begin{remark}\label{Remark:Kappa}
An elegant upper bound for the residue $\kappa_{\K}$ is due to Louboutin \cite{Louboutin00}:
\begin{equation}\label{eq:LouboutinKappa}\qquad
    \kappa_{\K} \leq \left(\frac{e\log{|\Delta_{\K}|}}{2(n_{\K} - 1)}\right)^{n_{\K} - 1}
    \quad \text{for $n_{\K} \geq 2$}.
\end{equation}
In Section \ref{sec:ResidueBounds}, we give several unconditional lower bounds on $\kappa_{\K}$. First, there is 
\begin{equation*}
\kappa_{\K} \geq \frac{0.36232}{\sqrt{|\Delta_{\K}|}}.
\end{equation*}
For $n_{\K} \geq 3$, an analysis of Stark's paper \cite{StarkBS} yields
the asymptotically better bound
\begin{equation*}
    \kappa_{\K} > \frac{0.0014480}{n_{\K} g(n_{\K}){|\Delta_{\K}|}^{1/n_{\K}}},
\end{equation*}
in which $g(n_{\K})=1$ if $\K$ has a normal tower over $\Q$ and $g(n_{\K}) = n_{\K}!$ otherwise.
However, there are concerns about a constant employed in his proof; see Remark \ref{Remark:Pintz}.
Section \ref{sec:ResidueBounds} contains improvements in special cases
and additional digits of accuracy.
\end{remark}

\subsection*{Outline of the paper}
Section \ref{Section:Proof} contains the proof of Theorem \ref{Theorem:Main},
which occupies the bulk of the paper.
In Section \ref{sec:ResidueBounds}, we obtain the unconditional, explicit lower bounds for $\kappa_{\K}$ 
discussed in Remark \ref{Remark:Kappa}.
We conclude in Section \ref{Section:Future} with remarks and future questions.

\subsection*{Acknowledgements}
We thank Matteo Bordignon, Korneel Debaene, Tristan Freiberg, Eduardo Friedman, and Tim Trudgian for their feedback and suggestions.
Thanks also to Joshua Suh and Jiahui Yu for double checking our computations. 
Finally, special thanks to the anonymous referee for many detailed comments.

\section{Proof of Theorem \ref{Theorem:Main}}\label{Section:Proof}

We split the proof of Theorem \ref{Theorem:Main} across several subsections.  
In what follows, $f(x) = O^{\star}(g(x))$ means $|f(x)| \leq |g(x)|$ for all $x$ in a pre-defined range (often
$x \geq 2$).
This is similar to Landau's big-$O$ notation, except the implied constant is always $1$.
To begin, we require some preliminary remarks.

\subsection{Preliminaries}

Fix a number field $\K$ with $n_{\K}\geq 2$ and ring of integers $\O_{\K}$. 
Let $I_{\K}(n)$ denote the number of ideals with norm $n$ and let $P_{\K}(n)$ denote the number of prime ideals in $\O_{\K}$ with norm $n$.
Borevich--Shafarevich \cite[p. 220]{Borevich} tells us that if $p^k$ is a rational prime power, then $I_{\K}(p^k)\leq (k+1)^{n_{\K}}$.
The total multiplicativity of the norm means that a non-prime ideal may have norm $p^k$, so one might suspect that a tighter bound
can be found for $P_{\K}(p^k)$. This expectation is well founded.

If $\PP\subset\mathcal{O}_{\K}$ is a prime ideal, then it divides exactly one rational prime $p$ and $N(\mathfrak{p}) = p^k$ for some $1\leq k\leq n_{\K}$ \cite[Thm.~5.14c]{StewartTall}.
Moreover, $p\O_{\K}$ has a unique factorization
\begin{equation*}
p\O_{\K} = \PP_1^{e_1} \cdots \PP_r^{e_r}
\end{equation*}
into prime ideals $\PP_i$, where $e_i \in \N$ is the ramification index of $\PP_i$.
The $\PP_i$ are the only prime ideals in $\K$ with norm equal to a power of $p$.
In fact, $N(\PP_i) = p^{f_i}$, in which the inertia degrees $f_i$ satisfy $f_i\leq n_{\K}$ and
\begin{equation*}
e_1 f_1 + \cdots + e_r f_r = n_{\K}.
\end{equation*}
In particular, for each rational prime $p$ the corresponding inertia degrees satisfy
\begin{equation}\label{eq:InertiaSum}
    {\sum_{f_i}} f_i \leq n_{\K}\quad\text{hence}\quad P_{\K}(p^k)\leq \bigg\lfloor \frac{n_{\K}}{k} \bigg\rfloor \leq \frac{n_{\K}}{k}.
\end{equation}

We require the following technique to obtain estimates for sums over prime ideals.
Suppose $g$ is a nonnegative arithmetic function and recall that a prime ideal $\PP$ with $N(\PP) \leq x$ lies over exactly one rational prime $p \leq x$.  Then
\begin{equation*}
    G(x)
    = \sum_{N(\PP)\leq x} g(N(\PP))
    \leq \sum_{p\leq x}{\sum_{f_i}} g(p^{f_i}),
\end{equation*}
in which $\sum_{f_i}$ denotes the sum over the inertia degrees $f_i$ of the prime ideals lying over $p$. 
If one can apply \eqref{eq:InertiaSum}, the previous sum can be simplified.  For example,
\begin{equation*}
    \theta_{\K}(x)
    = \sum_{N(\PP)\leq x} \log{N(\PP)} 
    \leq \sum_{p\leq x}\sum_{f_i} \log{p^{f_i}}
    = \sum_{p\leq x}\bigg(\sum_{f_i} f_i\bigg) \log{p}
    \leq n_{\K}\theta_{\Q}(x),
\end{equation*}
in which $\theta_{\Q}$ denotes the Chebyshev theta function and $\theta_{\K}$ its number-field analogue.

Finally, to avoid the problems which might arise from an exceptional zero of $\zeta_{\K}$, we introduce the summatory function 
\begin{equation*}
\I_{\K}(x) = \sum_{n\leq x} I_{\K}(n).
\end{equation*}
This is the number-field analogue of the integer-counting function $\lfloor x \rfloor$. 
Our proof relies on the following unconditional result of Sunley.

\begin{theorem}[Sunley]\label{Theorem:Sunley}
Let $\K$ be a number field with $n_{\K} \geq 2$.  For $x > 0$,
\begin{equation}\label{eq:Sunley}
\I(x) = \kappa_{\K} x + O^{\star} ( \Lambda_{\K} x^{1 - \frac{2}{n_{\K}+1}}),
\end{equation}
in which
\begin{equation*}
\Lambda_{\K} = e^{28.2n_{\K}+5} (n_{\K}+1)^{ \frac{5(n_{\K}+1)}{2}} |\Delta_{\K}|^{\frac{1}{n_{\K}+1}} 
(\log |\Delta_{\K}|)^{n_{\K}}.
\end{equation*}
\end{theorem}

This result is \cite[Thm.~2]{SunleyBAMS}, although the range of
admissible $x$ is not specified and a proof is not given (this is common for short
research announcements in the \emph{Bulletin of the AMS} like this).
Sunley's result also appears as \cite[Thm.~1.1]{SunleyTAMS}, again without proof or an explicit range of
admissible $x$.  Consequently, we were forced to go back to Sunley's doctoral
thesis, in which the result is originally proved \cite{SunleyThesis}.

The desired result is stated, with no mention of the range of admissible $x$, 
as the first part of \cite[Thm.~1, p.~17]{SunleyThesis} and restated as \cite[Thm.~3.3.5, p.~54]{SunleyThesis}.
For convenience, and to verify that Sunley intended \eqref{eq:Sunley} to apply for $x>0$, we  examine the proof presented her thesis \cite{SunleyThesis}.
The proof begins at the bottom of p.~54 with the consideration of (in our notation) the first case
$0 < x \leq 2 n_{\K} \sqrt{|\Delta_{\K}|}$; this clearly indicates that Sunley intended \eqref{eq:Sunley}
to be taken for $x>0$.  Let
\begin{align*}
    a_1 = e^{28.2n_{\K}+5} (n_{\K}+1)^{ \frac{5(n_{\K}+1)}{2}},\quad
    a_3 = 2^{2n_{\K}} e^{\frac{1}{2}} \pi^{n_{\K}} (1.3)^{n_{\K}+1}, \quad
    a_7 = 2^{4n_{\K}+2} 5^{n_{\K}} {n_{\K}}!
\end{align*}
These constants are defined on \cite[p.~54, 20, 28]{SunleyThesis}, respectively.

In \cite[Lem.~3.1.1]{SunleyThesis}, Sunley notes that $\kappa_{\K} \leq a_3 (\log{|\Delta_{\K}|})^{n_{\K}}$.
In \cite[Thm.~3.1.6]{SunleyThesis}, Sunley proves that
\begin{equation*}
    | \I_{\K}(x) | \leq
    \begin{cases}
    (\log{|\Delta_{\K}|})^{n_{\K} - 1}x & \text{for $0 \leq x < 2$} ,\\[5pt]
    n_{\K} \binom{n_{\K}-1}{\left\lfloor (n_{\K}-1)/2 \right\rfloor}(\log{|\Delta_{\K}|})^{n_{\K} - 1} x& \text{for $2 \leq x \leq |\Delta_{\K}|$} ,\\[5pt]
    a_7 (\log{|\Delta_{\K}|})^{n_{\K} - 1} x& \text{for $x > |\Delta_{\K}|$} ,
    \end{cases}
\end{equation*}
in which $\binom{n_{\K}-1}{\left\lfloor (n_{\K}-1)/2 \right\rfloor}$ is a binomial coefficient.
This case is therefore dealt with on \cite[p.~55]{SunleyThesis} using the preceding estimates in the following way:
\begin{align*}
    | \I_{\K}(x) - \kappa_{\K} x |
    &\leq | \I_{\K}(x) | + \kappa_{\K} x \\
    &\leq (a_7 + a_3) (\log |\Delta_{\K}|)^{n_{\K}-1} x^{1 - \frac{2}{n_{\K}+1}} x^{\frac{2}{n_{\K}+1}}\\
    &\leq (a_7 + a_3) (2n_{\K})^{\frac{2}{n_{\K}+1}} |\Delta_{\K}|^{\frac{1}{n_{\K}+1}} (\log |\Delta_{\K}|)^{n_{\K}} x^{1 - \frac{2}{n_{\K}+1}}.
\end{align*}
Now, one can verify that 
\begin{equation}\label{eqn:minor_edits_time}
    (a_7 + a_3) (2n_{\K})^{\frac{2}{n_{\K}+1}} \leq a_1
\end{equation}
for $n_{\K} \geq 1$, so the first case of Sunley's theorem holds.
The case $x > 2n_{\K}\sqrt{|\Delta_{\K}|}$ is handled using complex analysis and a ``moving the line of integration'' argument, but the end result replicates \eqref{eq:Sunley}, as expected. In particular, the $a_1$ term arises during this aspect of the proof, and this is the reason one does not need a stricter upper bound in \eqref{eqn:minor_edits_time}.
It follows that \eqref{eq:Sunley} holds for $x>0$.

\begin{remark}
If $\K= \Q$, then $\I_{\Q}(x) = x + O^{\star}(1)$. However, $\I_{\Q}(x) = \lfloor x \rfloor \leq x$ is more precise.
This is one reason Rosser--Schoenfeld obtain better error estimates in Mertens' theorems for $\K = \Q$ \cite{Rosser};  the fact that the Riemann zeta function has no exceptional zero provides them 
more options as well.
\end{remark}

\subsection{Preparatory lemmas to prove \eqref{eq:A1} and \eqref{eq:A2}}

Before we establish \eqref{eq:A1} and \eqref{eq:A2} in Section \ref{Section:Proof:A}, we need 
several technical lemmas to estimate
\begin{equation*}
    \sum_{N(\PP)\leq x} \I_{\K} \bigg(\frac{x}{N(\PP)}\bigg) \log{N(\PP)}\qquad\text{and}\qquad\sum_{N(\PP)\leq x} \log{N(\PP)} \sum_{j\geq 2} \I_{\K} \left(\frac{x}{N(\PP^j)}\right).
\end{equation*}
We need the following result of Rosser--Schoenfeld \cite[Thm.~9]{Rosser}: 
\begin{equation}\label{eq:Rosser}
    \theta(x) = \sum_{p\leq x}\log p  < 1.01624x < 1.1x,
    \quad \text{for $x>0$}.
\end{equation}
A recent improvement on \eqref{eq:Rosser} yields smaller constants
throughout; see Remark \ref{Remark:Theta}.

\begin{lemma}\label{Lemma:Painful}
    For $\alpha \geq 0$ and $x \geq 2$, 
    \begin{equation*}
        \sum_{p\leq x} \frac{\log p}{p^{\alpha}}
        <
        \begin{cases}
            \dfrac{1.1}{1-\alpha}x^{1-\alpha} & \text{if $0 \leq \alpha \leq 1$},\\[10pt]
            \log x & \text{if $\alpha=1$},\\[5pt]
            \dfrac{1.1\,\alpha}{(\alpha - 1)2^{\alpha-1}}  & \text{if $\alpha > 1$}.
        \end{cases}
    \end{equation*}
\end{lemma}

\begin{proof}
    Rosser--Schoenfeld \cite[(3.24)]{Rosser} established the result for $\alpha = 1$.
    Suppose $x \geq 2$. 
    For $\alpha > 0$ with $\alpha \neq 1$, partial summation and \eqref{eq:Rosser} yield
    \begin{align*}
        \sum_{p\leq x} \frac{\log p}{p^{\alpha}}
        &= \frac{\theta(x)}{x^{\alpha}} + \alpha\int_2^x \frac{\theta(t)}{t^{\alpha+1}}\,dt 
        < 1.1\left( \frac{1}{x^{\alpha-1}} + \alpha \int_2^x \frac{dt}{t^{\alpha}} \right) \\
        &=
        \begin{cases}
            1.1 \left( x^{1-\alpha} + \dfrac{\alpha}{1-\alpha} \left( x^{1-\alpha} - 2^{1-\alpha} \right) \right)  & \text{if $0<\alpha < 1$},\\[10pt]
            1.1 \left(\dfrac{1}{x^{\alpha-1}} + \dfrac{\alpha}{\alpha - 1} \left( \dfrac{1}{2^{\alpha-1}} - \dfrac{1}{x^{\alpha-1}} \right) \right)  & \text{if $\alpha > 1$},\\
        \end{cases}
    \end{align*}
    which implies the desired result for $\alpha\neq 1$.
\end{proof}

The preceding lemma and some computation yield the next lemma.

\begin{lemma}\label{Lemma:Multipart}
    For $j \in \N$, $n_{\K}\geq 2$, and $x \geq 2$,
    \begin{equation*}
        x^{1 - \frac{2}{n_{\K}+1}} \!\!\! \sum_{N(\PP)\leq x} \frac{\log N(\PP)}{N(\PP)^{j (1- \frac{2}{n_{\K}+1})}}
        <
        \begin{cases}
       0.55\,n_{\K}(n_{\K}+1)x  & \text{if $j =1$ or $j=n_{\K}=2$},\\[5pt] 
        n_{\K}x^{ 1 - \frac{2}{n_{\K}+1} }  \log x & \text{if $(j,n_{\K})= (2,3)$ or $(3,2)$},\\[5pt]
        \dfrac{13.2\, n_{\K} x^{1 - \frac{2}{n_{\K}+1}} }{2^{j /3}}& \text{otherwise}.
    \end{cases}
\end{equation*}
\end{lemma}

\begin{proof}
First observe that \eqref{eq:InertiaSum} implies
\begin{align*}
\sum_{N(\PP)\leq x} \frac{\log N(\PP)}{N(\PP)^{j (1- \frac{1}{n_{\K}})}}
&\leq \sum_{p\leq x} \sum_{f_i} \frac{\log (p^{f_i}) }{ p^{f_i j (1- \frac{2}{n_{\K}+1} )} }   \\
&\leq \sum_{p\leq x} \sum_{f_i} f_i \frac{\log p }{ p^{j (1- \frac{2}{n_{\K}+1}) } }   \\
&\leq n_{\K} \sum_{p\leq x}\frac{\log p }{ p^{j (1- \frac{2}{n_{\K}+1}) } }  ,  
 \end{align*}
in which $\sum_{f_i}$ denotes the sum over the inertia degrees $f_i$ of the prime ideals lying over the rational prime $p$. Next substitute 
\begin{equation}\label{eq:alpha}
\alpha = j  \bigg(1- \frac{2}{n_{\K}+1} \bigg)
\end{equation}
into Lemma \ref{Lemma:Painful}, multiply by $x^{1 - \frac{2}{n_{\K}+1}}$, and obtain
\begin{equation} \label{eq:ContinueBound}
x^{1 - \frac{2}{n_{\K}+1}} \!\!\! \sum_{N(\PP)\leq x} \frac{\log N(\PP)}{N(\PP)^{j (1- \frac{2}{n_{\K}+1})}}
\leq n_{\K} x^{1 - \frac{2}{n_{\K}+1}} \sum_{p\leq x}\frac{\log p }{ p^{j (1- \frac{2}{n_{\K}+1}) } } .
\end{equation}
Refer to Table \ref{Table:Jane} in the case-by-case analysis below.
    
    \begin{table}
    \begin{equation*}
    \begin{array}{c|ccccccccccccc}
    j \backslash n_{\K} & 2 & 3 & 4 & 5 & 6 & 7 & 8 & 9 & 10 & 11 & 12 & 13 & 14\\[5pt]
    \hline
    1 & \color{green!50!black}\frac{1}{3} &\color{green!50!black} \frac{1}{2} &\color{green!50!black} \frac{3}{5} &\color{green!50!black} \frac{2}{3} &\color{green!50!black} \frac{5}{7} &\color{green!50!black} \frac{3}{4} &\color{green!50!black} \frac{7}{9} &\color{green!50!black} \frac{4}{5} &\color{green!50!black} \frac{9}{11} &\color{green!50!black} \frac{5}{6} &\color{green!50!black} \frac{11}{13} &\color{green!50!black} \frac{6}{7} &\color{green!50!black} \frac{13}{15} \\[5pt]
    2&\color{green!50!black}\frac{2}{3} &\color{red} 1 &\color{blue} \color{blue}\frac{6}{5} &\color{blue} \frac{4}{3} &\color{blue} \frac{10}{7} &\color{blue} \frac{3}{2} &\color{blue} \frac{14}{9} &\color{blue} \frac{8}{5} &\color{blue} \frac{18}{11} &\color{blue} \frac{5}{3} &\color{blue} \frac{22}{13} &\color{blue} \frac{12}{7} &\color{blue} \frac{26}{15} \\[5pt]
    3&\color{red}1 &\color{blue} \frac{3}{2} &\color{blue} \frac{9}{5} &\color{blue} 2 &\color{blue} \frac{15}{7} &\color{blue} \frac{9}{4} &\color{blue} \frac{7}{3} &\color{blue} \frac{12}{5} &\color{blue} \frac{27}{11} &\color{blue} \frac{5}{2} &\color{blue} \frac{33}{13} &\color{blue} \frac{18}{7} &\color{blue} \frac{13}{5} \\[5pt]
    4&\color{blue} \frac{4}{3} &\color{blue} 2 &\color{blue} \frac{12}{5} &\color{blue} \frac{8}{3} &\color{blue} \frac{20}{7} &\color{blue} 3 &\color{blue} \frac{28}{9} &\color{blue} \frac{16}{5} &\color{blue} \frac{36}{11} &\color{blue} \frac{10}{3} &\color{blue} \frac{44}{13} &\color{blue} \frac{24}{7} &\color{blue} \frac{52}{15} \\[5pt]
    5&\color{blue} \frac{5}{3} &\color{blue} \frac{5}{2} &\color{blue} 3 &\color{blue} \frac{10}{3} &\color{blue} \frac{25}{7} &\color{blue} \frac{15}{4} &\color{blue} \frac{35}{9} &\color{blue} 4 &\color{blue} \frac{45}{11} &\color{blue} \frac{25}{6} &\color{blue} \frac{55}{13} &\color{blue} \frac{30}{7} &\color{blue} \frac{13}{3} \\[5pt]
    6&\color{blue} 2 &\color{blue} 3 &\color{blue} \frac{18}{5} &\color{blue} 4 &\color{blue} \frac{30}{7} &\color{blue} \frac{9}{2} &\color{blue} \frac{14}{3} &\color{blue} \frac{24}{5} &\color{blue} \frac{54}{11} &\color{blue} 5 &\color{blue} \frac{66}{13} &\color{blue} \frac{36}{7} &\color{blue} \frac{26}{5} \\[5pt]
    7&\color{blue} \frac{7}{3} &\color{blue} \frac{7}{2} &\color{blue} \frac{21}{5} &\color{blue} \frac{14}{3} &\color{blue} 5 &\color{blue} \frac{21}{4} &\color{blue} \frac{49}{9} &\color{blue} \frac{28}{5} &\color{blue} \frac{63}{11} &\color{blue} \frac{35}{6} &\color{blue} \frac{77}{13} &\color{blue} 6 &\color{blue} \frac{91}{15} \\
    \end{array}
    \end{equation*}
    \caption{Values of $\alpha = j(1-\frac{2}{n_{\K}+1})$ for $j \geq 1$ and $n_{\K} \geq 2$.  Values with $\alpha < 1$ are in {\color{green!50!black}green},
    $\alpha =1$ in {\color{red}red}, and $\alpha >1$ in {\color{blue}blue}.}
    \label{Table:Jane}
    \end{table}

If $j =1$, or if $j=n_{\K}=2$, then $0 < \alpha < 1$ and \eqref{eq:ContinueBound} can be bounded from above by the first case of Lemma \ref{Lemma:Painful}:
\begin{align*}
n_{\K} x^{1 - \frac{2}{n_{\K}+1}} \sum_{p\leq x} \frac{\log p}{p^{j (1- \frac{2}{n_{\K}+1})}}
&< n_{\K}  x^{1 - \frac{2}{n_{\K}+1}} \frac{1.1}{1 - (1- \frac{2}{n_{\K}+1})} x^{1 - (1- \frac{2}{n_{\K}+1})}\\
&= 0.55\,n_{\K}(n_{\K}+1)x.
\end{align*}

If $j=2$ and $n_{\K}=3$, or if $j=3$ and $n_{\K}=2$, then $\alpha = 1$.
In these two cases, the second case of Lemma \ref{Lemma:Painful} immediately yields the desired upper bound.

Otherwise, $\alpha > 1$ and we are in the third case of Lemma \ref{Lemma:Painful}. We must maximize
\begin{equation*}
f(\alpha) = \frac{\alpha}{\alpha -1}
\end{equation*}
over all pairs $(j,n_{\K})$ shown in blue in Table \ref{Table:Jane}.  
Observe that $f'(\alpha)=-(\alpha-1)^{-2}<0$, and hence $f(\alpha)$ decreases for $\alpha > 1$.
Therefore, we must minimize $\alpha$ over all admissible pairs $(j,n_{\K})$.
The definition \eqref{eq:alpha} ensures that $\alpha$ increases in both $j$ and $n_{\K}$, so the desired
maximum can be found by inspection of Table \ref{Table:Jane}.
The maximum value of $f$ occurs at $(j,n_{\K}) = (2,4)$, for which $\alpha = \frac{6}{5}$ and $f(\alpha) = 6$.
Since $n_{\K} \geq 2$, the third case of Lemma \ref{Lemma:Painful} implies 
\begin{align*}
n_{\K}x^{1 - \frac{2}{n_{\K}+1}} \sum_{p\leq x} \frac{\log p}{p^{j (1- \frac{2}{n_{\K}+1})}} 
&< \dfrac{1.1\, n_{\K} \alpha x^{1 - \frac{2}{n_{\K}+1}}}{ \big(\alpha - 1 \big)2^{j (1- \frac{2}{n_{\K}+1})-1}} \\
&\leq \dfrac{6.6\, n_{\K}x^{1 - \frac{2}{n_{\K}+1}}}{ 2^{j (1- \frac{2}{n_{\K}+1})-1}} \\
&\leq \frac{13.2\, n_{\K} x^{1 - \frac{2}{n_{\K}+1}} }{2^{j /3}}. \qedhere
\end{align*}
\end{proof}

Our next two lemmas are estimates obtained with the aid of Lemma \ref{Lemma:Multipart}.
The first one is rather straightforward, but the second is much more involved.

\begin{lemma}\label{Lemma:FirstPower}
For $x\geq2$,
\begin{equation*}
\sum_{N(\PP)\leq x} \I_{\K} \bigg(\frac{x}{N(\PP)}\bigg) \log{N(\PP)} 
= \kappa_{\K} x \sum_{N(\PP)\leq x} \frac{\log{N(\PP)}}{N(\PP)} 
+ 0.55\, \Lambda_{\K}n_{\K}(n_{\K}+1) O^{\star} (x) .
\end{equation*}
\end{lemma}

\begin{proof}
Theorem \ref{Theorem:Sunley} and Lemma \ref{Lemma:Multipart} with $j  = 1$ imply
\begin{align*}
&\sum_{N(\PP)\leq x} \I_{\K} \bigg(\frac{x}{N(\PP)}\bigg) \log{N(\PP)}\\
&\qquad= \sum_{N(\PP)\leq x} \Bigg( \kappa_{\K}  \bigg(\frac{x}{N(\PP)}\bigg) 
+ O^{\star}\bigg(\Lambda_{\K} \bigg(\frac{x}{N(\PP)}\bigg)^{1-\frac{2}{n_{\K}+1}} \bigg) \Bigg) \log N(\PP)  \\
&\qquad= \kappa_{\K} x \!\!  \sum_{N(\PP)\leq x} \frac{\log{N(\PP)}}{N(\PP)} 
+ \Lambda_{\K} O^{\star}\Bigg(x^{1 - \frac{2}{n_{\K}+1}}\sum_{N(\PP)\leq x}  
\frac{ \log N(\PP)}{N(\PP)^{1-\frac{2}{n_{\K}+1}} } \Bigg)  \\
&\qquad= \kappa_{\K} x \!\!  \sum_{N(\PP)\leq x} \frac{\log{N(\PP)}}{N(\PP)} 
+ 0.55\, \Lambda_{\K}n_{\K}(n_{\K}+1) O^{\star} (x)  .\qedhere
\end{align*}
\end{proof}

\begin{lemma}\label{Lemma:HigherPowers}
For $x\geq 2$,
\begin{equation*}
\sum_{N(\PP)\leq x} \log{N(\PP)} \sum_{j\geq 2} \I_{\K} \left(\frac{x}{N(\PP^j)}\right) 
\,=\, O^{\star}(\Xi_{\K}(x)),
\end{equation*}
in which
\begin{equation}\label{eq:Xi}
\Xi_{\K}(x)=
\begin{cases}
\kappa_{\K} n_{\K} x + \Lambda_{\K} O^{\star}( 3.3\,x + 2x^{\frac{1}{3}}\log x + 50.8\,x^{\frac{1}{3}}  )  & \text{if $n_{\K} = 2$}, \\[5pt]
\kappa_{\K} n_{\K} x + \Lambda_{\K} O^{\star}( 3\,x^{ \frac{1}{2}}  \log x +  96\, x^{\frac{1}{2}}  )  & \text{if $n_{\K} = 3$},\\[5pt]
\kappa_{\K} n_{\K} x + \Lambda_{\K} O^{\star}( 40.31 \,n_{\K}x^{1 - \frac{2}{n_{\K}+1}}  )  & \text{if $n_{\K} \geq 4$}. 
\end{cases}
\end{equation}
\end{lemma}

\begin{proof}
Theorem \ref{Theorem:Sunley} and the total multiplicativity of the norm imply that
\begin{align}
&\sum_{N(\PP)\leq x} \log{N(\PP)} \sum_{j\geq 2} \I_{\K} \left(\frac{x}{N(\PP^j)}\right) \\
&\quad= \sum_{N(\PP)\leq x} \log{N(\PP)} \sum_{j\geq 2} \Bigg( \frac{\kappa_{\K} x}{N(\PP^j)} 
+ O^{\star}\bigg( \Lambda_{\K}\bigg(\frac{x}{N(\PP^j)}\bigg)^{1 - \frac{2}{n_{\K}+1}}  \bigg) \Bigg) \nonumber \\
&\quad= \kappa_{\K} x  \underbrace{ \sum_{N(\PP)\leq x} \sum_{j\geq 2} \frac{\log{N(\PP)}}{N(\PP)^j}  }_{\text{Term 1}}
+ \Lambda_{\K}  O^{\star} \Bigg(
\underbrace{ x^{1 - \frac{2}{n_{\K}+1}}  \!\!\!  \sum_{N(\PP)\leq x}\sum_{j\geq 2} \frac{\log{N(\PP)} }{N(\PP)^{j(1 - \frac{2}{n_{\K}+1})}}  }_{\text{Term 2}}
 \Bigg) .\label{eq:Terms}
\end{align}

\medskip\noindent\textit{Term 1.}
Use \eqref{eq:InertiaSum} to obtain
\begin{align*}
    \sum_{N(\PP)\leq x} \sum_{j\geq 2} \frac{\log{N(\PP)}}{N(\PP)^j}
    &\leq  \sum_{p\leq x} \sum_{f_i} \sum_{j\geq 2} \frac{\log (p^{f_i})}{(p^{f_i})^j} \\
    &\leq  \sum_{p\leq x} \sum_{j\geq 2} \bigg(\sum_{f_i}f_i\bigg) \frac{\log{p}}{p^{j}} \\
    &\leq  n_{\K}  \sum_{p} \sum_{j\geq 2} \frac{\log{p}}{p^{j}} \\
    &= n_{\K}  \sum_{p}\frac{\log{p}}{p(p-1)} \\
    &< n_{\K},
\end{align*}
in which $\sum_{f_i}$ denotes the sum over the inertia degrees $f_i$ of the prime ideals lying over the
rational prime $p$ (the final sum is bounded above by $0.7554$). 

\medskip\noindent\textit{Term 2.}
Apply Lemma \ref{Lemma:Multipart} and obtain
\begin{align*}
&x^{1 - \frac{2}{n_{\K}+1}}\sum_{N(\PP)\leq x}  \sum_{j\geq 2} \frac{\log{N(\PP)} }{N(\PP)^{j(1 - \frac{2}{n_{\K}+1})}}  \\
&\qquad = \sum_{j\geq 2}  \bigg( x^{1 - \frac{2}{n_{\K}+1}}\sum_{N(\PP)\leq x}  \frac{\log{N(\PP)} }{N(\PP)^{j(1 - \frac{2}{n_{\K}+1})}} \bigg) \\
&\qquad = 
\begin{cases}
\displaystyle 0.55\, n_{\K}(n_{\K}+1)x + n_{\K}x^{1 - \frac{2}{n_{\K}+1}}\log x + 13.2\,n_{\K}x^{1 - \frac{2}{n_{\K}+1}}\sum_{j \geq 4} \frac{1}{2^{j/3}}& \text{if $n_{\K}=2$},\\[5pt]
\displaystyle n_{\K}x^{1 - \frac{2}{n_{\K}+1}} \log x +  13.2\,n_{\K}x^{1 - \frac{2}{n_{\K}+1}}\sum_{j\geq 3} \frac{1}{2^{j/3}}& \text{if $n_{\K}=3$},\\[5pt]
\displaystyle 13.2\,n_{\K}x^{1 - \frac{2}{n_{\K}+1}}\sum_{j\geq 2} \frac{1}{2^{j/3}}& \text{if $n_{\K}\geq 4$},
\end{cases}
\\
&\qquad \leq 
\begin{cases}
\displaystyle 3.3\,x + 2x^{\frac{1}{3}}\log x + 50.8\,x^{\frac{1}{3}}  & \text{if $n_{\K}=2$},\\[5pt]
\displaystyle 3\,x^{ \frac{1}{2}}  \log x +  96\, x^{\frac{1}{2}} & \text{if $n_{\K}=3$},\\[5pt]
\displaystyle 40.31 \,n_{\K}x^{1 - \frac{2}{n_{\K}+1}} & \text{if $n_{\K}\geq 4$},
\end{cases}
\end{align*}
To complete the proof, return to \eqref{eq:Terms} and use the estimates above.
\end{proof}

\subsection{Proof of \eqref{eq:A1} and \eqref{eq:A2}}\label{Section:Proof:A}

Consider
\begin{equation*}
T_{\K}(x) = \log\bigg(\prod_{N(\mathfrak{a})\leq x}N(\mathfrak{a})\bigg) = \sum_{n\leq x} I_{\K}(n)\log{n},
\end{equation*}
in which $\mathfrak{a}\subset\O_{\K}$ runs over the integral ideals of $\K$.
In the next two lemmas, we approximate $T_{\K}(x)$ in two different ways.  Comparing the resulting expressions will complete the proof of \eqref{eq:A1}. 
The following lemma is an explicit version of Weber's theorem, which states $T_{\K}(x) = \kappa_{\K} x \log x + O(x)$ \cite[p.~128]{DiamondHalberstam}.

\begin{lemma}\label{Lemma:FirstBound}
For $x \geq 2$,
\begin{equation*}
T_{\K}(x) =  \kappa_{\K} x \log{x} + \left(  \frac{(n_{\K}+1)^2}{2(n_{\K}-1)}\Lambda_{\K}   +\kappa_{\K}\right) O^{\star}(x).
\end{equation*}
\end{lemma}

\begin{proof}
Partial summation and Theorem \ref{Theorem:Sunley} imply 
\begin{align}
T_{\K}(x)
&= \sum_{n\leq x}I_{\K}(n)\log{n} = \sum_{2\leq n\leq x}I_{\K}(n)\log{n}\nonumber\\
&= \I_{\K}(x)\log{x} - \int_{2}^{x}\frac{\I_{\K}(t)}{t}dt\nonumber\\
&= \big(\kappa_{\K} x + O^{\star}(\Lambda_{\K}x^{1-\frac{2}{n_{\K}+1}})\big) \log{x} \nonumber\\
&\qquad\quad +O^{\star}\bigg( \int_{2}^{x}\frac{\kappa_{\K} t}{t}\,dt + \int_{2}^{x}\frac{\Lambda_{\K}t^{1-\frac{2}{n_{\K}+1}}}{t}\,dt \bigg).\label{eqn:comebacktome}
\end{align}
Calculus reveals that $\log{x} \leq \alpha x^{1/\alpha}$ for $x \geq 1$ and $\alpha > 0$.  
Let $\alpha = \frac{n_{\K}+1}{2}$ and deduce
\begin{equation*}
\log x <  \tfrac{1}{2}( n_{\K}+1) x^{ \frac{2}{n_{\K}+1}}.
\end{equation*}
Therefore,
\begin{equation*}
    \big(\kappa_{\K} x + O^{\star}(\Lambda_{\K}x^{1-\frac{2}{n_{\K}+1}})\big) \log{x}
    = \kappa_{\K} x \log{x} + \tfrac{1}{2}\Lambda_{\K} ( n_{\K}+1 )O^{\star}(x).
\end{equation*}
Since $n_{\K} \geq 2$,
\begin{equation*}
    \int_2^x t^{-\frac{2}{n_{\K}+1}}dt
     =\left.\frac{n_{\K}+1}{n_{\K} - 1}t^{\frac{n_{\K}-1}{n_{\K}+1}}\right|_{2}^{x} 
     < \frac{n_{\K}+1}{n_{\K} - 1}x^{\frac{n_{\K}-1}{n_{\K}+1}} 
     \leq \frac{n_{\K}+1}{n_{\K} - 1} x.
\end{equation*}
Return to \eqref{eqn:comebacktome} and observe that
\begin{equation*}
	\int_{2}^{x}\frac{\kappa_{\K} t}{t}\,dt +\int_{2}^{x}\frac{\Lambda_{\K}t^{1-\frac{2}{n_{\K}+1}}}{t}\,dt 
    = \bigg(\kappa_{\K}+\frac{n_{\K}+1}{n_{\K} - 1}\Lambda_{\K} \bigg) O^{\star}(x).
\end{equation*}
Put this all together, recall that $n_{\K}\geq 2$, and obtain
\begin{align*}
T_{\K}(x) 
&=  \kappa_{\K} x \log{x} + \tfrac{1}{2}\Lambda_{\K} ( n_{\K}+1 )O^{\star}(x)
+ \bigg(\kappa_{\K}+\frac{n_{\K}+1}{n_{\K} - 1}\Lambda_{\K} \bigg) O^{\star}(x)\\
&= \kappa_{\K} x \log{x} +
 \left(  \frac{(n_{\K}+1)^2}{2(n_{\K}-1)}\Lambda_{\K}   +\kappa_{\K}\right) O^{\star}(x). \qedhere
\end{align*}
\end{proof}

Now, we derive a second explicit approximation for $T_{\K}(x)$.

\begin{lemma}\label{Lemma:SecondBound}
For $x \geq 2$,
\begin{equation*}
T_{\K}(x) = \kappa_{\K} x \sum_{N(\PP)\leq x} \frac{\log{N(\PP)}}{N(\PP)} +  0.55\, \Lambda_{\K}n_{\K}(n_{\K}+1) O^{\star} (x)+ O^{\star}(\Xi_{\K}(x)),
\end{equation*}
in which $\Xi_{\K}(x)$ is given by \eqref{eq:Xi}.
\end{lemma}

\begin{proof}
We require the ideal analogue of the Legendre--Chebyshev identity\footnote{Diamond--Halberstam \cite[p.~128]{DiamondHalberstam}
inform us that Landau calls this the ``Poincar\'e identity.''} \cite{Legendre},
\begin{equation*}
\prod_{N(\mathfrak{a})\leq x}N(\mathfrak{a}) = \prod_{N(\PP)\leq x}
\prod_{j\geq 1}N(\PP)^{\I_{\K} ( x / N(\PP^j) )}.
\end{equation*}
Theorem \ref{Theorem:Sunley} and Lemma \ref{Lemma:Multipart} with $j  = 1$ imply
\begin{align*}
&T_{\K}(x)
= \log\bigg(\prod_{N(\mathfrak{a})\leq x}N(\mathfrak{a})\bigg) 
= \log\bigg(\prod_{N(\PP)\leq x} \prod_{j\geq 1}N(\PP)^{\I_{\K} ( x / N(\PP^j) )} \bigg) \\
&\,\,=\sum_{N(\PP)\leq x} \log{N(\PP)} \sum_{j\geq 1} \I_{\K} \bigg(\frac{x}{N(\PP^j)}\bigg)\\
&\,\,=\sum_{N(\PP)\leq x} \I_{\K} \bigg(\frac{x}{N(\PP)}\bigg) \log{N(\PP)}  +\sum_{N(\PP)\leq x} \log{N(\PP)} \sum_{j\geq 2} \I_{\K} \left(\frac{x}{N(\PP^j)}\right)  \\
&\,\,= \underbrace{\kappa_{\K} x \sum_{N(\PP)\leq x} \frac{\log{N(\PP)}}{N(\PP)} + 0.55\, \Lambda_{\K}n_{\K}(n_{\K}+1) O^{\star} (x) }_{\text{Lemma \ref{Lemma:FirstPower}}}+
\underbrace{O^{\star}(\Xi_{\K}(x))}_{\text{Lemma \ref{Lemma:HigherPowers}}} . \qedhere
\end{align*}
\end{proof}

We are now in a position to complete the proof of \eqref{eq:A2}.
Equate the two expressions for $T_{\K}(x)$ from Lemmas \ref{Lemma:FirstBound} and \ref{Lemma:SecondBound} and deduce
\begin{align*}
& \kappa_{\K} x \log{x} + \left(  \frac{(n_{\K}+1)^2}{2(n_{\K}-1)}\Lambda_{\K}   +\kappa_{\K}\right) O^{\star}(x) \\
&\quad = \kappa_{\K} x \sum_{N(\PP)\leq x} \frac{\log{N(\PP)}}{N(\PP)}  + 0.55\, \Lambda_{\K}n_{\K}(n_{\K}+1) O^{\star} (x) + O^{\star}(\Xi_{\K}(x)).
\end{align*}
Divide by $\kappa_{\K} x$, simplify, and get
\begin{align*}
\sum_{N(\PP)\leq x} \frac{\log{N(\PP)}}{N(\PP)}
&= \log x +   \left(  \frac{(n_{\K}+1)^2}{2\kappa_{\K}(n_{\K}-1)}\Lambda_{\K}   + 1\right) O^{\star}(1) \\
&\qquad + \frac{0.55\, \Lambda_{\K}n_{\K}(n_{\K}+1)}{\kappa_{\K}} O^{\star} (1) + O^{\star}\bigg(\frac{\Xi_{\K}(x)}{\kappa_{\K}x}\bigg).
\end{align*}
From \eqref{eq:Xi} observe that
\begin{align}
\frac{\Xi_{\K}(x)}{\kappa_{\K}x} 
&=\frac{1}{\kappa_{\K}x} \cdot
\begin{cases}
\kappa_{\K} n_{\K} x + \Lambda_{\K} O^{\star}( 3.3\,x + 2x^{\frac{1}{3}}\log x + 50.8\,x^{\frac{1}{3}}  )  & \text{if $n_{\K} = 2$}, \\[5pt]
\kappa_{\K} n_{\K} x + \Lambda_{\K} O^{\star}( 3\,x^{ \frac{1}{2}}  \log x +  96\, x^{\frac{1}{2}}  )  & \text{if $n_{\K} = 3$},\\[5pt]
\kappa_{\K} n_{\K} x + \Lambda_{\K} O^{\star}( 40.31 \,n_{\K}x^{1 - \frac{2}{n_{\K}+1}}  )  & \text{if $n_{\K} \geq 4$}. 
\end{cases}
\nonumber\\
&=
\begin{cases}
2 + \cfrac{\Lambda_{\K}}{\kappa_{\K}} O^{\star}( 3.3 + 2x^{-\frac{2}{3}}\log x + 50.8\,x^{-\frac{2}{3}}  )  & \text{if $n_{\K} = 2$}, \\[8pt]
3 + \cfrac{\Lambda_{\K}}{\kappa_{\K}} O^{\star}( 3\,x^{ -\frac{1}{2}}  \log x +  96\, x^{-\frac{1}{2}}  )  & \text{if $n_{\K} = 3$},\\[8pt]
n_{\K}  + \dfrac{\Lambda_{\K}}{\kappa_{\K}} O^{\star}( 40.31 \,n_{\K}x^{- \frac{2}{n_{\K}+1}}  )  & \text{if $n_{\K} \geq 4$},
\end{cases}
\nonumber\\
&=
\begin{cases}
n_{\K} + \cfrac{\Lambda_{\K}}{\kappa_{\K}} O^{\star}(36.18  )  & \text{if $n_{\K} = 2$}, \\[8pt]
n_{\K} + \cfrac{\Lambda_{\K}}{\kappa_{\K}} O^{\star}( 69.36  )  & \text{if $n_{\K} = 3$},\\[8pt]
n_{\K}  + \dfrac{\Lambda_{\K}}{\kappa_{\K}} O^{\star}( 40.31 \,n_{\K}  )  & \text{if $n_{\K} \geq 4$},
\end{cases}
\label{eq:BetterBit}
\\
&=O^{\star}\bigg( n_{\K} +  40.31 \dfrac{\Lambda_{\K}n_{\K}}{\kappa_{\K}} \bigg). \nonumber
\end{align}
Put this all together and obtain
\begin{equation*}
\sum_{N(\PP)\leq x} \frac{\log{N(\PP)}}{N(\PP)}  = \log x +  A_{\K}(x),
\end{equation*}
in which
\begin{equation*}
|A_{\K}(x)| \leq
     \underbrace{  \left(  \frac{(n_{\K}+1)^2}{2\kappa_{\K}(n_{\K}-1)}\Lambda_{\K}   + 1\right)
+  \frac{0.55\, \Lambda_{\K}n_{\K}(n_{\K}+1)}{\kappa_{\K}} + n_{\K} + 
 40.31 \dfrac{\Lambda_{\K}n_{\K}}{\kappa_{\K}} }_{\Upsilon_{\K}}.
\end{equation*}
This yields the desired bound \eqref{eq:A2}. 
\qed

\subsection{Proofs of \eqref{eq:B1} and \eqref{eq:B2}}\label{Section:Proof:B}
For $x \geq 2$, partial summation yields
\begin{align*}
    \sum_{N(\PP)\leq x}&\frac{1}{N(\PP)}
    = \sum_{N(\PP)\leq x}\frac{\log{N(\PP)}}{N(\PP)}\frac{1}{\log{N(\PP)}}\\
    &= \frac{1}{\log{x}} \sum_{N(\PP)\leq x}\frac{\log{N(\PP)}}{N(\PP)} + \int_2^x \bigg(\sum_{N(\PP)\leq t}\frac{\log{N(\PP)}}{N(\PP)}\bigg)\frac{dt}{t(\log t)^2}\\
    &= \frac{1}{\log{x}} \big(\log{x} + A_{\K}(x)\big) + \int_2^x \big(\log{t} + A_{\K}(t)\big)\frac{dt}{t(\log t)^2}\\
    &= 1 + \frac{A_{\K}(x)}{\log{x}} + \int_2^x \frac{dt}{t\log t} + \int_2^x \frac{A_{\K}(t)}{t(\log t)^2}dt\\
    &= \log\log{x} - \log\log{2}  + 1 + \frac{A_{\K}(x)}{\log{x}} + \int_2^x \frac{A_{\K}(t)}{t(\log t)^2}dt\\
    &= \log\log{x} + \underbrace{1- \log\log{2}  + \int_2^\infty \frac{A_{\K}(t)}{t(\log t)^2}dt}_{M_{\K}} 
    +  \underbrace{ \frac{A_{\K}(x)}{\log{x}} - \int_x^\infty \frac{A_{\K}(t)}{t(\log t)^2}dt }_{B_{\K}(x)},
\end{align*}
in which \eqref{eq:A2} ensures that the integral that defines $M_{\K}$ converges and
\begin{equation*}
    |B_{\K}(x)|
    \leq \frac{|A_{\K}(x)|}{\log{x}} + \int_x^\infty \frac{|A_{\K}(t)|}{t(\log t)^2}dt
    \leq  \Upsilon_{\K} \bigg( \frac{1}{\log x} + \int_x^\infty \frac{dt}{t(\log t)^2}\bigg)
    = \frac{2\Upsilon_{\K}}{\log x}.
\end{equation*}
This proves \eqref{eq:B1} and \eqref{eq:B2}. \qed

\subsection{Proof of \eqref{eq:M1}}
Now we must find the constant $M_{\K}$; our approach is based on Ingham's \cite{Ingham}.  Define
\begin{align*}
    g(s) = \sum_{\PP}\frac{1}{N(\PP)^s} = \lim_{x\rightarrow \infty}\sum_{N(\PP)\leq x}\frac{1}{N(\PP)^s},
\end{align*}
which is analytic on $\Re s > 1$.
For $x \geq 2$, partial summation implies
\begin{align*}
    \sum_{N(\PP)\leq x}\frac{1}{N(\PP)^s}
    &= \sum_{N(\PP)\leq x}\frac{1}{N(\PP)} N(\PP)^{1-s}\\
    &= \frac{1}{x^{s - 1}}\sum_{N(\PP)\leq x}\frac{1}{N(\PP)} + (s-1) \int_{2}^x\bigg(\sum_{N(\PP)\leq t}\frac{1}{N(\PP)}\bigg)\frac{dt}{t^{s}}.
\end{align*}
Since $\Re(s-1) >0$ and
\begin{equation*}
\sum_{N(\PP)\leq x}\frac{1}{N(\PP)} = \log \log x + O(1)
\end{equation*}
by \cite[Lem.~2.4]{Rosen}, it follows that
\begin{equation*}
\lim_{x\to\infty}\frac{1}{x^{s - 1}}\sum_{N(\PP)\leq x}\frac{1}{N(\PP)} = 0.
\end{equation*}
Then for $\Re s > 1$,
\begin{align*}
g(s) 
&= (s-1) \int_{2}^{\infty}\bigg(\sum_{N(\PP)\leq t}\frac{1}{N(\PP)}\bigg)\frac{dt}{t^{s}} \\
&= (s-1)\int_2^{\infty} \big( \log \log t + M_{\K} + B_{\K}(t) \big) \frac{dt}{t^s} \\
&=  \underbrace{ (s-1)\int_{2}^\infty\frac{M_{\K}}{t^{s}}\,dt}_{I_1(s)} 
+ \underbrace{(s-1)\int_{2}^\infty\frac{B_{\K}(t)}{t^{s}} \,dt }_{I_2(s)}
+\underbrace{(s-1)\int_{2}^\infty\frac{\log\log t}{t^{s}}\,dt}_{I_3(s)} .
\end{align*}

\noindent\textit{First Integral.}
First observe that
\begin{equation*}
\lim_{s\to 1^+} I_1(s) = M_{\K} \lim_{s\to 1^+} \left((s-1)\int_2^{\infty} \frac{dt}{t^s}\right)
= M_{\K} \lim_{s\to 1^+} 2^{1-s} = M_{\K}.
\end{equation*}

\noindent\textit{Second Integral.}
From \eqref{eq:B2}, we have
\begin{equation*}
\frac{|B_{\K}(t)|}{t^s} = O\left(\frac{1}{t^s \log t}\right).
\end{equation*}
Let $u = \log t$, so that $du = dt/t$ and $e^u = t$, and conclude that
\begin{align}
(s-1) \int_2^{\infty} \frac{dt}{t^s \log t}
&=  (s-1) \int_{(s-1)\log 2}^{\infty} \frac{e^{-v}}{v }\,dv\nonumber \\
&= -(s-1) \Ei\big( (1-s)\log 2\big), \label{eq:EiEiO}
\end{align}
in which
\begin{equation*}
\Ei(x) = - \int_{-x}^{\infty} \frac{e^{-t}}{t}\,dt 
\end{equation*}
is the exponential integral function (the singularity is handled in the Cauchy principal value sense).
Since $\Ei(x) =  \log x + O(1)$ as $x \to 0^+$, \eqref{eq:EiEiO} ensures that
\begin{equation*}
\lim_{s\to 1^+} I_2(s) = 0.
\end{equation*}
Alternatively, one can avoid the exponential integral by using the identity
\begin{align*}
    \int_{z}^{\infty} \frac{e^{-v}}{v }\,dv 
    =  - \log{z} +\int_{z}^{1} \frac{e^{-v} - 1}{v}\,dv + \int_{1}^{\infty} \frac{e^{-v}}{v }\,dv 
    \qquad \text{for $z>0$}.
\end{align*}

\noindent\textit{Third Integral.}
Using the substitution $t^{s-1} = e^y$, we obtain
\begin{align*}
    I_3(s)
    &= \int_{\log(2^{s-1})}^{\infty} e^{-y} \log y \, dy - 2^{1-s} \log(s-1).
\end{align*}
Recalling that
\begin{equation*}
    \gamma  =  -\int_0^{\infty}e^{-t} \log t \,dt ,
\end{equation*}
we conclude that $I_3(s) = - \gamma- \log(s-1) + o(1)$ as $s \to 1^+$.

Putting this all together, 
$g(s) =  M_{\K} - \gamma -\log(s-1) + o(1)$
and hence
\begin{equation}\label{eq:LggaK}
M_{\K} = \gamma+\log(s-1) + g(s) +o(1) 
\end{equation}
as $s \to 1^+$.  The Euler product formula for $\zeta_{\K}(s)$ ensures that
\begin{align*}
&\log(s-1) + g(s) =\log(s-1) + \sum_{\PP}\frac{1}{N(\PP)^s}\\
&\qquad=\log(s-1) + \sum_{\PP}\left[ \frac{1}{N(\PP)^s} + \log\left(1 - \frac{1}{N(\PP)^s}\right) \right] - \sum_{\PP}  \log\left(1 - \frac{1}{N(\PP)^s}\right) \\
&\qquad= \log\big((s-1)\zeta_{\K}(s)\big) + \sum_{\PP}\left[ \frac{1}{N(\PP)^s} + \log\left(1 - \frac{1}{N(\PP)^s}\right) \right] ,
\end{align*}
in which the sum is uniformly convergent by comparison with $\sum_{\PP} N(\PP)^{-2}$.
Since $\zeta_{\K}(s)$ has a simple pole at $s=1$ with residue $\kappa_{\K}$,
we conclude from \eqref{eq:LggaK} that
\begin{equation}\label{eq:AlphaGamma}
M_{\K} = \gamma + \log \kappa_{\K} + \sum_{\PP}\left[\frac{1}{N(\PP)} + \log\left(1 - \frac{1}{N(\PP)}\right)\right].
\end{equation}
This concludes the proof of \eqref{eq:M1}.\qed

\subsection{Proofs of \eqref{eq:C1} and \eqref{eq:C2}}
From \eqref{eq:AlphaGamma} we deduce
\begin{equation}\label{eq:refmelastX}
    -\gamma - \log{\kappa_{\K}} + M_{\K} 
    = \sum_{N(\PP) \leq x}\left[\frac{1}{N(\PP)} + \log\left(1 - \frac{1}{N(\PP)}\right)\right] + F_{\K}(x),
\end{equation}
in which
\begin{equation}\label{eq:EK}
    F_{\K}(x) = \sum_{N(\PP) > x}\left[\frac{1}{N(\PP)} + \log\left(1 - \frac{1}{N(\PP)}\right)\right].
\end{equation}
For $y \in [0,1)$, observe that
\begin{equation}\label{eq:yyy}
0 \leq -y -\log(1-y) \leq  \frac{y^2}{1 - y}.
\end{equation}
Let $y = 1 / N(\PP)$ and deduce
\begin{align}
    |F_{\K}(x)|
    &=- \sum_{N(\PP) > x}\left[\frac{1}{N(\PP)} + \log\left(1 - \frac{1}{N(\PP)}\right)\right] \nonumber\\
    &\leq \sum_{N(\PP) > x}\frac{1}{N(\PP) (N(\PP) - 1)} \nonumber\\
    &\leq \sum_{p > x}\sum_{f_i}\frac{1}{p^{f_i} (p^{f_i} - 1)} \nonumber\\
    &\leq \sum_{p > x}\bigg(\sum_{f_i}1\bigg)\frac{1}{p (p - 1)}\nonumber\\
    &< n_{\K}\sum_{m > x}\frac{1}{m (m - 1)} \nonumber\\
    &= \frac{n_{\K}}{ \lceil x \rceil-1 } \label{eq:nKm} \\
    &\leq \frac{n_{\K}}{ x-1 },\nonumber
\end{align}
in which ${\sum_{f_i}}$ denotes the sum over the inertia degrees $f_i$ of the prime ideals lying over $p$
and $\lceil x \rceil$ denotes the least integer greater than or equal to $x$.
In light of \eqref{eq:B1}, the right-hand side of \eqref{eq:refmelastX} becomes
\begin{align}
    \sum_{N(\PP) \leq x}&\left[\frac{1}{N(\PP)} + \log\left(1 - \frac{1}{N(\PP)}\right)\right] + F_{\K}(x)\nonumber\\
    &= \sum_{N(\PP) \leq x} \frac{1}{N(\PP)} + \sum_{N(\PP) \leq x}\log\left(1 - \frac{1}{N(\PP)}\right) + F_{\K}(x)\nonumber\\
    &= \sum_{N(\PP) \leq x} \log\left(1 - \frac{1}{N(\PP)}\right) + \log\log{x} + M_{\K} + E_{\K}(x),\label{eq:importanteX}
\end{align}
in which $E_{\K}(x) = F_{\K}(x) + B_{\K}(x)$.  
Exponentiate \eqref{eq:refmelastX} and use \eqref{eq:importanteX} to obtain
\begin{align*}
\frac{e^{-\gamma}e^{M_{\K}}}{\kappa_{\K}}
= \exp\bigg[\sum_{N(\PP) \leq x} \log\left(1 - \frac{1}{N(\PP)}\right)\bigg](\log x)
e^{M_{\K}} e^{E_{\K}(x)},
\end{align*}
which yields
\begin{equation*}
    \prod_{N(\PP) \leq x} \left(1 - \frac{1}{N(\PP)}\right) = \frac{e^{- \gamma}}{\kappa_{\K} \log{x}} e^{-E_{\K}(x)}.
\end{equation*}
Write
\begin{equation*}
C_{\K}(x) = e^{-E_{\K}(x)} - 1
\end{equation*}
and use the inequality $|e^t - 1|  \leq |t| e^{|t|}$, valid for $t \in \R$, to deduce that
\begin{equation*}
\prod_{N(\PP) \leq x} \left(1 - \frac{1}{N(\PP)}\right) = \frac{e^{- \gamma}}{\kappa_{\K} \log{x}} \big(1 + C_{\K}(x) \big),
\end{equation*}
in which
\begin{equation*}
|C_{\K}(x)| \leq  |E_{\K}(x)| e^{ |E_{\K}(x)|}.
\end{equation*}
This concludes the proof of \eqref{eq:C1}.\qed

\subsection{Proof of \eqref{eq:M2}}
From \eqref{eq:AlphaGamma}, we have
\begin{equation*}
M_{\K} = \gamma + \log\kappa_{\K} + F_{\K}(2-\delta)\quad \text{for $\delta \in (0,1)$},
\end{equation*}
in which 
$F_{\K}(x)$ is defined by \eqref{eq:EK}.  In particular, \eqref{eq:yyy} and \eqref{eq:nKm} reveal that
\begin{equation*}
- n_{\K} \leq \liminf_{\delta\to0^+}F_{\K}(2-\delta) \leq \limsup_{\delta\to0^+}F_{\K}(2-\delta)\leq 0.
\end{equation*}
Thus, $-n_{\K} \leq M_{\K} - \gamma - \log \kappa_{\K} \leq 0$,
which is equivalent to \eqref{eq:M2}.\qed

\section{Explicit lower bounds for the Dedekind-zeta residue}\label{sec:ResidueBounds}
For a number field $\K$, recall that $\kappa_{\K}$ denotes the residue of the Dedekind zeta function
$\zeta_{\K}(s)$ at $s=1$.  If $\K = \Q$, then $\kappa_{\K} = 1$.  Consequently, we assume 
that $n_{\K} \geq 2$. Although $\kappa_{\K}$ can be computed directly from the analytic class 
number formula \eqref{eqn:residue_class_ana_form}, it is worth investigating bounds 
that are given only in terms of the absolute value of the discriminant $\Delta_{\K}$
and the degree $n_{\K}$ of $\K$.

Since $n_{\K} = r_1 + 2r_2\geq 2$, it follows that $2^{r_1}(2\pi)^{r_2} \geq 2^{2}(2\pi)^{0} = 4$.
Friedman \cite[Thm.~B]{Friedman89} established that $R_{\K}/w_{\K} \geq 0.09058$,
a sharper version of a bound from Zimmert \cite{Zimmert} (see also \cite[Thm.~7, p.~273]{Lang}).
Consequently,
\begin{equation}\label{eqn:kapp_bound_Zimmert}
    \kappa_{\K} \geq \frac{2^{r_1}(2\pi)^{r_2}R_{\K}}{w_{\K}\sqrt{|\Delta_{\K}|}} > \frac{4\cdot 0.09058}{\sqrt{|\Delta_{\K}|}} = \frac{0.36232}{\sqrt{|\Delta_{\K}|}}.
\end{equation}

Another approach is based on Stark's estimate
\begin{equation}\label{eqn:kappd_bound_Stark}
    \kappa_{\K} > \frac{c}{n_{\K} g(n_{\K}){|\Delta_{\K}|}^{1/n_{\K}}},
\end{equation}
in which
\begin{equation*}
    g(n_{\K}) = \begin{cases}
    1 &\text{if }\K\text{ has a normal tower over }\Q,\\
    n_{\K} ! &\text{otherwise,}
    \end{cases}
\end{equation*}
and $c$ is effectively computable \cite{StarkBS}. The denominator in \eqref{eqn:kappd_bound_Stark} can be replaced by $g(n_{\K})\log{|\Delta_{\K}|}$ if $\K$ has no quadratic subfield. 
We show that $c = 0.001448029$ is likely unconditionally admissible (see Remark \ref{Remark:Pintz}) in \eqref{eqn:kappd_bound_Stark}, with improvements possible in some cases. If $n_{\K} > 2$, then \eqref{eqn:kappd_bound_Stark} is generally preferred over \eqref{eqn:kapp_bound_Zimmert}.

In what follows, we adhere to Stark's notation so that the reader may, if they wish,
confirm our calculations.
Stark proves the existence of effectively computable constants $c_4$ and $c_8$ such that
\begin{equation}
    \kappa_{\K} > {c_4}^{-1}\min\left\{\frac{1}{\alpha(n_{\K})\log{|\Delta_{\K}|}}, \frac{1}{c_8{|\Delta_{\K}|}^{1/n_{\K}}}\right\},
\end{equation}
in which
\begin{equation*}
    \alpha(n_{\K})=
    \begin{cases}
    4&\text{if }\K\text{ is normal over }\Q,\\
    16&\text{if }\K\text{ has a normal tower over }\Q,\\
    4n_{\K}!&\text{otherwise.}
    \end{cases}
\end{equation*}
Moreover, $\kappa_{\K} > 1/(c_4\alpha(n_{\K})\log{|\Delta_{\K}|})$ if $\K$ does not have a quadratic subfield. Therefore, in the case $\K$ does not have a quadratic subfield, 
\begin{align*}
    \kappa_{\K}
    &> \frac{1}{c_4c_8 {|\Delta_{\K}|}^{1/n_{\K}}} \min\left\{\frac{c_8{|\Delta_{\K}|}^{1/n_{\K}}}{\alpha(n_{\K})\log{|\Delta_{\K}|}}, 1\right\}\\
    &\geq \frac{1}{c_4c_8 {|\Delta_{\K}|}^{1/n_{\K}}} \min\left\{\frac{e c_8}{n_{\K}\alpha(n_{\K})}, 1\right\},
\end{align*}
since $e^x \geq x e$ for $x>0$ implies that 
$${|\Delta_{\K}|}^{1/n_{\K}} = \exp((1/n_{\K})\log{|\Delta_{\K}|}) \geq (e/n_{\K})\log{|\Delta_{\K}|}.$$ 

We find admissible constants $c_4$ and $c_8$ by carefully studying \cite{StarkBS}. First,
\begin{equation*}
    c_4 = 2c_3 = 2e^{\frac{21}{8} + \frac{c_1}{2} - \frac{c_2}{8}\frac{\Gamma'}{\Gamma}\left(\frac{1}{2}\right)},
\end{equation*}
in which $c_1 = 0$ and $c_2 = 2/\log{3}$ are admissible options \cite[Lemma 4]{StarkBS}. Hence, $c_4 \approx 43.162115 < 43.2$ is admissible. 
Stark claims somewhat mysteriously that ``it is likely from a remark in Bateman and Grosswald \cite[p.~188]{batemanGrosswald} that $c_8 = \pi/6$ will suffice"; see Remark \ref{Remark:Pintz}.
If one proceeds with $c_8 = \pi/6$, then 
\begin{equation*}
    \kappa_{\K}
    > \frac{1}{c_4c_8 {|\Delta_{\K}|}^{1/n_{\K}}} \min\left\{\frac{ec_8}{n_{\K}\alpha(n_{\K})}, 1\right\}
    = \frac{e/c_4}{n_{\K}\alpha(n_{\K}){|\Delta_{\K}|}^{1/n_{\K}}}
    > \frac{0.06297842}{n_{\K}\alpha(n_{\K}){|\Delta_{\K}|}^{1/n_{\K}}},
\end{equation*}
since $(ec_8)/(n_{\K}\alpha(n_{\K})) \leq (ec_8)/8 < 0.178 < 1$. It follows that
\begin{equation*}
    \kappa_{\K} >
    \begin{cases}
    \frac{0.015744605}{n_{\K}g(n_{\K}){|\Delta_{\K}|}^{1/n_{\K}}}&\text{if $\K$ is normal over $\Q$},\\[5pt]
    \frac{0.003936151}{n_{\K}g(n_{\K}){|\Delta_{\K}|}^{1/n_{\K}}}&\text{if $\K$ has a normal tower over $\Q$},\\[5pt]
    \frac{0.015744605}{n_{\K}g(n_{\K}){|\Delta_{\K}|}^{1/n_{\K}}}&\text{otherwise}.
    \end{cases}
\end{equation*}
Moreover, if $\K$ does not have a quadratic subfield, then
\begin{equation*}
    \kappa_{\K}
    > \frac{1/c_4}{\alpha(n_{\K})\log{|\Delta_{\K}|}}
    > \begin{cases}
    \frac{0.005792116}{g(n_{\K})\log{|\Delta_{\K}|}}&\text{if $\K$ is normal over $\Q$},\\[5pt]
    \frac{0.001448029}{g(n_{\K})\log{|\Delta_{\K}|}}&\text{if $\K$ has a normal tower over $\Q$},\\[5pt]
    \frac{0.005792116}{g(n_{\K})\log{|\Delta_{\K}|}}&\text{otherwise}.
    \end{cases}
\end{equation*}
Assuming that $c_8 = \pi/6$ is feasible, it follows that $c = 0.001448029$ is unconditionally admissible in \eqref{eqn:kappd_bound_Stark}, with improvements available if more information is known about $\K$. This justifies the claims made in Remark \ref{Remark:Kappa}.

\begin{remark}\label{Remark:Pintz}
Stark suggests that $c_8 = \pi/6$ is admissible in \cite[Lem.~11]{StarkBS} and Pintz \cite[Thm.~3]{Pintz76} proved that $c_8 = \pi/12+  o(1)$ works.  
This suggests that Stark's $c_8 = \pi / 6$ is acceptable.
Further evidence was provided by Schinzel, a referee of Pintz's paper, who improved the value of $c_8$ in a footnote to Pintz' paper \cite[p.~277]{Pintz76}.
That is, for each $\varepsilon > 0$ and $|\Delta_{\K}|$
sufficiently large, he proved that $c_8 = (16/\pi - \varepsilon)^{-1}$ is admissible \cite[p.~277]{Pintz76}.
Moreover, one can always just use the lower bound 
\eqref{eqn:kapp_bound_Zimmert}.
\end{remark}

\section{Remarks and Open Problems}\label{Section:Future}

\begin{remark}\label{Remark:Theta}
The estimate $\theta(x) < 1.01624x$ from \eqref{eq:Rosser} has been improved over the years.
The current record appears to be due to Broadbent, Kadiri, Lumley, Ng, and Wilk \cite{Broadbent}, who proved that
\begin{equation*}
\theta(x) < (1+1.93378 \times 10^{-8})x, \quad \text{for $x\geq 0$}.
\end{equation*}
This bound results in a slight improvement to the constants in
Theorem \ref{Theorem:Main}.
\end{remark}

\begin{question}
Tenenbaum \cite{Tenenbaum} recently proved a generalization of Mertens' second theorem.
Following similar arguments \textit{mutatis mutandis} it may be possible to write
\begin{equation*}
    \mathcal{S}_{\K}(k,x) = \sum_{N(\PP_1\PP_2\cdots\PP_k)\leq x}\frac{1}{N(\PP_1\PP_2\cdots\PP_k)} = P_{\K}(k,\log\log{x}) + O\left(\frac{(\log\log{x})^{k-1}}{\log{x}}\right),
\end{equation*}
for $x\geq 3$, in which $P_{\K}(k,X) = \sum_{0\leq j\leq k} \lambda_{j,k} X^j$ and $\lambda_{j,k}$ are defined as in \cite[Thm.~1]{Tenenbaum}.
Can one make the implied constant explicit in terms of $n_{\K}$ and $|\Delta_{\K}|$?
\end{question}

\begin{question}
In the case $\K = \Q$, Mertens' third theorem asserts
\begin{equation*}
    \prod_{p\leq x} \left(1 - \frac{1}{p}\right)^{-1} \,\sim\,\,\,e^\gamma \log{x} .
\end{equation*}
Rosser--Schoenfeld \cite{Rosser} observed that the product is less than $e^\gamma \log{x}$ for $x\leq 108$ and
they wondered if the two expressions took turns exceeding the other.
Diamond--Pintz proved that this is the case \cite{DiamondPintz}. In fact, the difference 
is $\Omega(\log\log\log{x}/\sqrt{x})$ infinitely often.
Does a similar phenomenon occur for $\K \neq \Q$?
\end{question}

\bibliographystyle{amsplain}
\bibliography{UEMTNF}

\end{document}